\documentclass[10pt]{amsart}
\usepackage{latexsym,fancyhdr,amssymb,color,amsmath,amsthm,graphicx,listings,comment}
\usepackage[section]{placeins}
\newtheorem{thm}{Theorem} 
\newtheorem{lemma}{Lemma} \newtheorem{coro}{Corollary}
\let\paragraph\subsection

\title{Coloring Discrete Manifolds}
\author{Oliver Knill}
\date{5/27/21}
\address{Department of Mathematics \\ Harvard University \\ Cambridge, MA, 02138 }
\makeindex

\begin{document}
\maketitle

\begin{abstract}
Discrete $d$-manifolds are classes of finite simple graphs which can triangulate classical
manifolds but which are defined entirely within graph theory.	
We show that the chromatic number $X(G)$ of a discrete $d$-manifold $G$ satisfies $d+1 \leq X(G) \leq 2(d+1)$.
From the general identity $X(A+B)=X(A)+X(B)$ for the join $A+B$ of two finite simple graphs, it follows that there are
$(2k)$-spheres with chromatic number $X=3k+1$ and $(2k-1)$-spheres with chromatic number $X=3k$. 
Examples of $2$-manifolds with $X(G)=5$ have been known since the pioneering work of Fisk.
Current data support the that an upper bound $X(G) \leq \lceil 3(d+1)/2 \rceil$ could hold for all $d$-manifolds $G$, 
generalizing a conjecture of Albertson-Stromquist \cite{AlbertsonStromquist}, stating $X(G) \leq 5$ for all $2$-manifolds. 
For a d-manifold, Fisk has introduced the $(d-2)$-variety $O(G)$. This graph $O(G)$ has maximal simplices of 
dimension $(d-2)$ and correspond to complete complete subgraphs $K_{d-1}$ of $G$ for which the dual 
circle has odd cardinality. In general, $O(G)$ is a union of $(d-2)$-manifolds. We note that if $O(S(x))$ is either empty 
or a $(d-3)$-sphere for all $x$ then $O(G)$ is a $(d-2)$-manifold or empty. 
The knot $O(G)$ is already interesting for $3$-manifolds $G$ because Fisk has demonstrated that every possible knot 
can appear as $O(G)$ for some $3$-manifold. For $4$-manifolds $G$ especially, the Fisk variety $O(G)$ is a 
$2$-manifold in $G$ as long as all $O(S(x))$ are either empty or a knot in every unit 3-sphere $S(x)$. 
\end{abstract}

\section{An overview}

\paragraph{}
A {\bf $d$-manifold} is a finite simple graph $G=(V,E)$ for which every {\bf unit sphere} $S(x)$, 
the graph generated by the neighbors of $x$, is a $(d-1)$-sphere. A {\bf $d$-sphere} is a $d$-manifold
for which one can remove a vertex $x$ and get a contractible graph $G-x$. A graph is {\bf contractible} if there exists a vertex $x$
such that its unit sphere $S(x)$ and the graph $G-x$ without this vertex are both contractible. 
Slating the {\bf empty graph} $0$ as the $(-1)$-sphere and the {\bf $1$-point graph} $1$ to be contractible form the
foundation of these inductive definitions. A $2$-manifold for example is a graph for which every unit sphere is a cyclic graph with $4$ or more vertices. 
What is the chromatic number $X$ of a discrete $d$-manifold? We prove the general 
bound $(d+1) \leq X \leq 2(d+1)$, where the lower bound $X=d+1$ is achieved for any {\bf Barycentric refined manifold} $G_1$
or any product manifold $G \times H$ defined as the Barycentric refinement of the Cartesian product of the Whitney complexes of $G$
and $H$. The upper bound $2d+2$ will follow from the fact that we can cover the dual graph $\hat{G}$ of a d-manifold $G$ with 
two disjoint forests $A,B$, where each is closed in $\hat{G}$ in the sense that its vertex set generates the graph within $G$. 
While for $d=1$, the bound $X \leq 3$ is obviously sharp. We do not know for any dimension $d>1$ whether the 
upper bound $2d+2$ is sharp. A natural conjecture is that in general $X(G) \leq \lceil 3(d+1)/2 \rceil$ 
for all $d$-manifolds $G$, where $\lceil x \rceil$ is the ceiling function. Still, $2d+2$ is also natural if one has an eye
on the classical Nash-Kuiper theorem for $C^1$-embeddings of compact Riemannian manifolds in Euclidean space which would
correspond to embed $G$ in a $(2d+1)$-manifold which can be minimally colored with $2d+2$ colors. 

\begin{figure}[h]
\scalebox{1.0}{\includegraphics{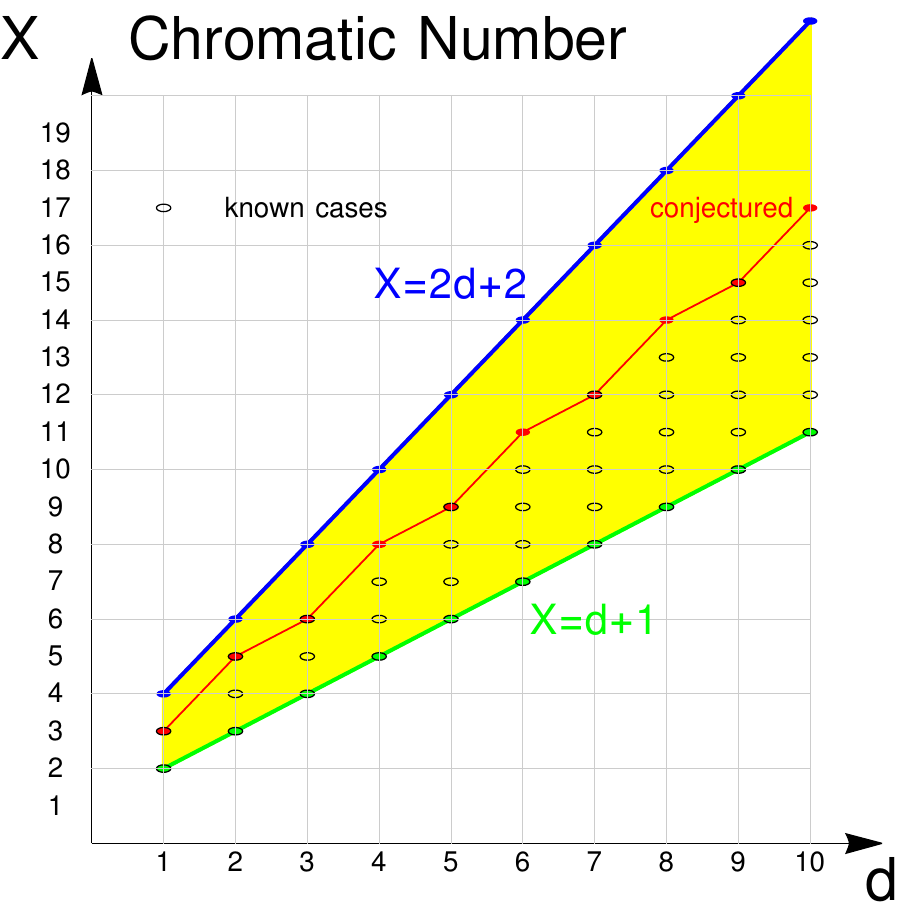}}
\caption{
This figure gives an overview of the chromatic number problem for $d$-manifolds.
This coloring problem is a graph theoretical question. Open problems are
to exclude $X=6$ for $2$-manifolds, to see whether $X=7$ is possible for 3-manifolds
and to find a $4$-manifold with $X=8$.
}
\end{figure}

\paragraph{}
By looking at unit spheres, we can see that for any topological type of $(d=2k)$-manifolds, there are 
examples with $X=3k+1$ and for any topological type of $(2k-1)$-manifolds there are examples with $X=3k$. 
We can find $(d=2k-1)$-spheres with $X=3k$ and $(d=2k)$-spheres with $X=3k+1$. 
Albertson and Stromquist \cite{AlbertsonStromquist}
conjectured $X\leq 5$ for $2$-manifolds (using different notions like locally planar but this should be 
equivalent). We know $X \leq 6$ for $2$-manifolds but do not know of an example yet of a $2$-manifold with $X=6$.
For $d=3$-manifolds, we only know $X \leq 8$ and have examples like $X(C_5 + C_5)=6$ but we have no example 
of a $3$-manifold with either $X=7$ or $X=8$. For $d=4$-manifolds we know $X \leq 10$, and $X(C_5 + H)=7$ 
if $H$ is a $2$-sphere with $X(H)=4$. We are not aware of any $4$-manifold $G$ yet with $X(G)=8$. 
If the formula $X(G) \leq \lceil 3(d+1)/2 \rceil$ were true and sharp, then there should exist a $4$-manifold 
with $X=8$. 

\paragraph{}
Because $2$-spheres are by Whitney characterized as maximally planar $4$-connected graphs
with at least $5$ vertices, the $4$-color theorem is equivalent to $X \leq 4$ for $2$-spheres. Coloring
a $2$-sphere should then be possible constructively:
realize the $2$-sphere $S$ under consideration as the boundary of a $3$-ball $B$ which after a sequence of 
interior edge-refinement $B=B_0,B_1,B_2,\dots,B_n$ which do not modify $S$ to have 3-Eulerian in the interior and so $4$-colorable. 
{\bf 3-Eulerian} in the interior means for 3-manifolds with boundary that every circle $S(a) \cap S(b)$ has even degree for 
every edge $(a,b)$ not contained in the boundary. 
3-Eulerian does not necessarily mean Eulerian in the sense of having all vertex degrees even.
The justification for the name is through coloring: a $2$-sphere has even vertex
degrees for all vertices if and only if it can be colored with the minimum of $3$ colors. 
A $3$-sphere has even edge degree for all edges if and
only if it can be colored with the minimum of $4$ colors. If $S$ is a $2$-sphere and $B=S+P_2$ is the suspension, a 
$3$-sphere, then the $2$-sphere $S$ is Eulerian if and only if the 3-sphere $B$ is 3-Eulerian. 
We could define a $d$-manifold to be {\bf $d$-Eulerian} in general as the property of having an empty Fisk variety $O(M)$. 

\paragraph{}
The {\bf Euler-Hierholzer theorem} assures that for an Eulerian $2$-sphere,
there is an {\bf Eulerian path}, a closed curve visiting every edge exactly once. On an Eulerian sphere, one has also a natural 
{\bf geodesic flow}, as we can continue a path naturally through every vertex. Sometimes, one even can get an {\bf ergodic} or transitive
geodesic flow in the sense that the geodesic is an Eulerian path \cite{knillgraphcoloring3}. In the case when $G$ is a non-Eulerian
$2$-manifold, one can always use a {\bf geodesic cutting algorithm} \cite{knillgraphcoloring3} 
to render it Eulerian. The reason why we can pair up points in $O(M)$ 
is that by the Euler handshake formula the number of odd degree vertices is always even in a $2$-manifold. 

\paragraph{}
This project is a continuation of \cite{KnillNitishinskaya,knillgraphcoloring,knillgraphcoloring2,knillgraphcoloring3}. 
The original spark \cite{KnillNitishinskaya} dealt mostly with the question when the minimal chromatic number $X=3$
is possible for $2$-manifolds. Literature search in December 2014 lead us to the pioneering work of Fisk \cite{Fisk1977b}
and his construction of tori with chromatic number $5$ and the definition of the odd part $O(M)$. 
We were originally interested in graphs with empty Fisk set $O(M)$ because of coloring
and especially because of the relation with the $4$-color theorem. The $4$-color theorem follows from the property of
being able to edge refine a $3$-ball so that $O(M)$ is confined to the boundary. 
So, in general, the topology and combinatorics of $O(M)$ within $M$ is interesting.

\paragraph{}
The coloring by manifold embedding prompted us to believe at first that a $d$-sphere can always be colored by $d+1$
or $d+2$ colors. As we have seen, this is can not be the case because there are 3-spheres with chromatic number $6$. 
Still, the relation with classical embedding problems produce interesting analogies: we know by the 
{\bf Whitney embedding theorem}
that every compact $d$-manifold can be embedded in a $2d$-Euclidean space suggesting that we should be able to color 
with $2d+1$ colors. There is still a possibility that this is actually the upper bound. For $d=2$ it gives $5$ for 
$d=3$ it gives $7$. Before trying to shoot for $X(G) \leq \lceil 3(d+1)/2 \rceil$ one could try to reach first
the easier $X(G) \leq 2d+1$, provided this is possible. 
If the embedding should decide about the chromatic number, how come that $d$-spheres
which can be embedded in a $(d+1)$-dimensional ball can not be colored by $d+2$ colors in general? 
Maybe it is not the Whitney embedding but the {\bf Nash-Kuiper embedding} which matters. We can by Nash-Kuiper 
$C^1$-embed any compact Riemannian $d$-manifold isometrically into a $(2d+1)$-dimensional manifold,
suggesting an upper chromatic bound $2d+2$ because we expect being able to locally 
refine the simply-connected ambient space
become minimally colorable with $2d+2$ colors. Now, if we
can refine an ambient $(2d+1)$-dimensional space in which $G$ is embedded without touching $G$, then $2d+2$ 
colors are sufficient. This is only an analogy of course as we do not use classical manifolds here.
Still in the light of embedding theorems, the upper bound $2d+2$ becomes a natural one. 

\section{Upper bound}

\paragraph{}
We now look for an upper bound on the chromatic number of {\bf $d$-manifolds}, finite simple graphs which have $(d-1)$-spheres 
as unit spheres $S(x)$. In general, the minimal coloring problem for simply connected
manifolds could be located in a polynomial complexity class, unlike the full problem on all graphs. 
The reason for being restricting to simply connected is that the manifold coloring problem in general depends on global 
properties and this could mean in general that the chromatic number might need full knowledge about the fundamental group.
Also the in general hard problem of constructing {\bf Hamiltonian paths} is easier for $d$-manifolds. 
We know that $d$-manifolds are Hamiltonian \cite{HamiltonianManifolds} a result which generalizes the statement of Whitney
in the case $d=2$. The recursive combinatorial definitions allow to use induction referring locally to unit spheres $S(x)$
which are manifolds themselves of one dimension less. 
The general problem of coloring graphs with a minimal number of colors is known to be an NP-complete problem.
Still, it is possible that on classes of $d$-manifolds the problem is easier. Especially on $2$-spheres, the coloring problem should be solvable effectively in polynomial time. If that is the case, we believe that the topological frame-work
could be a key in proving it. We especially expect that coloring 2-spheres and so solving the 4-color problem 
can be done in polynomial time. 

\paragraph{}
It is necessary to stress here that there is a similar sounding but different 
topological graph coloring problem for graphs on $2$-manifolds $G$ of 
genus $g(G)$ \cite{Ringel1974}. In this part of {\bf topological graph theory} 
\cite{TuckerGross}, the upper bound $\chi(G) \leq H(g(G))$ was proven by Heawood \cite{Heawood49},
who also conjectured it to be sharp.
Ringel and Youngs proved in \cite{RingelYoungs} the {\bf map color theorem}, stating 
that the {\bf Heawood number} $H(g)=[(7+\sqrt{48g+1})/2]$ is 
a sharp upper bound on the chromatic number of a graph $G$ embeddable in a surface of genus $g$ and $G$ is different 
from the Klein bottle (covered by Philip Franklin \cite{Fanklin1934}). 
The algebraic number in the floor bracket is the solution to the equation $(n-3)(n-4)/12=g$
and Ringel and Youngs needed to show that on a surface of genus $g=\lceil (n-3)(n-4)/12 \rceil$, there is a complete subgraph $K_n$
embedded. For $n=7$, this gives the torus case $g=1$ of Heawood.
The Heawood conjecture theme is a different problem than the manifold coloring problem we consider here. 
In the torus $g=1$ case for example, Heawood already saw
that one can embed the complete graph $K_7$ into a $g=1$ surface $T^2=R^2/Z^2$ but $K_7$ is a $6$-dimensional simplex. 
The $2$-manifolds considered here are finite simple graphs which do not even allowed to contain the complete graph $K_4$.  
The $d$-manifold coloring problem is defined in all dimensions. It invokes only a finite, graph theoretical frame-work 
without any need to use Euclidean space. There appears also to be no (at least no obvious) connection to the coloring problem of 
Sarkaria \cite{Sarkaria1981} which is a generalization of the Heawood topological graph theory set-up to higher dimensions. 

\paragraph{}
A {\bf conjecture of Albertson-Stromquist} \cite{AlbertsonStromquist} can be rephrased in the current context 
that $2$-manifolds can be colored by $5$ colors or less. We will just prove the upper bound $X \leq 6$ for $2$-manifolds. 
In the case $d=2$, one could also see that $X \leq 6$
by producing closed Kempe chains cutting the manifold into contractible parts, where each part is a planar graph 
having chromatic number $\leq 4$. When cutting, we have to make sure to have even length of all the Kempe dividing chains. 
This can always be done by making detours (which is possible as long as not all
vertex degrees are 4). The general inequality however can be proven with less effort and works in any dimension: 

\begin{thm}
Every d-manifold can be colored by by $X$ colors with $d+1 \leq X \leq 2d+2$.
\label{Theorem1}
\end{thm}

\begin{proof}
The {\bf dual graph} $\hat{G}=(\hat{V},\hat{E})$ has as vertex set
$\hat{V}$ the maximal simplices of $G$. Two maximal simplices are connected if they intersect in a 
$(d-1)$-dimensional simplex. 
This graph is {\bf triangle-free} as otherwise, a triangle would appear as the intersection of 
$(d-2)$ unit spheres $S(x_j)$, where $x=(x_0,\dots,x_{d-2})$ is a $(d-2)$-simplex in $G$. 
 While we will not need this, one can also note that the graph is 
{\bf $(d+1)$- regular} because every simplex is connected to exactly 
$d+1$ neighbors. We now can partition $\hat{V}$ into 
two forests, where each tree generates itself in $\hat{G}$ (if two vertices 
in the tree are connected in $\hat{G}$ then they must be connected in the tree).
The trees of the first forest use a coloring of the 
corresponding simplices with a first batch of $(d+1)$ colors, 
the second forest tells which maximal simplices in $G$ are colored with
the second batch of $(d+1)$ colors.
Since $\hat{G}$ is triangle free, the zero'th $H^0(G')$ and first 
cohomology $H^1(\hat{G})$ are the only cohomology groups of interest. 
Assume $\hat{G}$ has the Betti vector $(1,b)$ so that 
$\chi(G)=|V|-|E|=1-b$ is the Euler characteristic. 
Now cut $\hat{G}$ at $b$ places to break the $b$ homology 
cycles but making sure that the $b$ edges (which are to be cut) 
do not form a closed loop. This produces a tree $T \subset \hat{G}$. 
Color this tree with $2$ colors. Now put back 
the edges to get $\hat{G}$ and $\hat{G}$ is partitioned into two sets, 
where none of them has a closed loop of the same color. 
If there would be such a loop, it would have to 
intersect with no edge in $T$ and so consist entirely of edges which were cut. 
To justify that one can cut $\hat{G}$ in $b$ places without having a 
closed loop is proven by contradiction: assume we had a minimal 
example of a graph for which we need to make cuts along a closed loop 
to prune it to a tree. Then by minimality, each of the
attached loops also would have to consist entirely of cuts. The same 
argument now can be applied again to see that all loops 
would have to be cut at every edge. By minimiality the union of loops is
the graph. But that would mean that the
number of edges is equal to $b$.  This is impossible because by 
the {\bf Euler-Poincar\'e formula}, the number of edges $|E|=|V|+b-1$ 
is larger than $b$.  
\end{proof}

\paragraph{}
The cohomology is only involved to justify that we can cut a
triangle free graph $(V,E)$ with less than $|E|$ edges to 
render it a tree. One could also see this in an elementary fashion.
Cutting all edges except $1$ always produces a graph without loops.

\paragraph{}
This result prompts to look at a general {\bf homotopy coloring problem} for finite simple graphs: 
how many colors are needed, if each connected coloring patch with the same color needs to 
be a contractible connected component? An upper
bound is obviously the {\bf Lusternik-Schnirelman capacity}, the minimal number of contractible
graphs which cover the graph (see \cite{josellisknill} for a graph version in a frame-work like here). 
This is enough as we can just color each of these patches 
differently. This means that spheres and balls have the {\bf homotopy chromatic number} $2$. A $d$-dimensional torus
has Lusternik-Schnirelman capacity $d+1$ and so homotopy chromatic number bound above by $d+1$. 
Theorem~\ref{Theorem1} tells that for every graph without triangles,
this {\bf homotopy chromatic number} is either $1$ or $2$ and that the homotopy chromatic number is 
$1$ in the triangle-free case if and only if we deal with a forest. For triangle-free graphs, the
chromatic number can be arbitrarily large \cite{Mycielski} and that in the planar case, the chromatic number 
is $\leq 3$ by {\bf Groetsch's theorem}. 

\section{Arithmetic}

\paragraph{}
Here is a simple observation about the {\bf Zykov join} \cite{Zykov}
$A+B$ for general finite simple graphs $A,B$. (For simplicity, we write here $+$ rather 
than $\oplus$ as done in other places.) Remember that $A+B$ is
the disjoint union $A \cup B$ for which additionally every vertex in $A$ connected to 
every vertex in $B$. The join $A+B$ of two spheres is a sphere again
because $S_{A + B}(a) = S(a) + B$ if $a \in V(A)$ and 
$S_{A + B}(b) = A + S(b)$ if $b \in V(B)$. Inductively with 
respect to dimension, we have that both $A + S(b)$ and $S(a) + B$ are spheres, establishing
so that the join $A+B$ is a sphere. 

\begin{lemma}
The chromatic number is additive $X(A + B) = X(A) + X(B)$ for the join operation in graphs. 
\end{lemma}
\begin{proof}
The join $A + B$ is the disjoint union with the additional enhancement that 
all vertices in $A$ connected to all vertices in $B$.
The color set on the $A$ side therefore has to be disjoint with the color set on the $B$ side so 
that the chromatic number must be larger or equal. 
An explicit coloring shows that the chromatic number is the sum.
\end{proof}

\paragraph{}
This lemma immediately gives examples of spheres which have relatively large chromatic number:

\begin{coro} 
There are $(2k-1)$-spheres that can have chromatic number $3k$. 
There are $(2k)$-spheres with chromatic number $4+3(k-1)=3k+1$. 
\end{coro} 
\begin{proof}
An example is $C_5 + C_5 + \cdots + C_5$. The sum of $k$ such circles 
is a $(2k-1)$-sphere. For the second statement, throw in an additional 
$0$-sphere which increases the chromatic number by $1$. 
\end{proof}

\paragraph{}
A graph of maximal dimension $d$ is called {\bf minimally chromatic} if $X(G)=d+1$. 
In a simply connected $d$-manifold, this will be equivalent to the fact that the 
Fisk variety $O(M)$ is empty. Minimally chromatic graphs produce a sub-monoid of all graphs: 

\begin{coro}
If $A$ and $B$ are minimally chromatic graphs, then $A+B$ is minimally chromatic. 
\end{coro}
\begin{proof}
Both the {\bf clique number} $f={\rm dim}+1$ as well as the chromatic number $f=X$ are
additive $f(A+B)=f(A)+f(B)$.
\end{proof}

\paragraph{}
If $S_0=\{ \{a,b\},  \emptyset \}$ is the $0$-sphere, then $G + S_0$ is called the 
{\bf suspension} of $G$. This is completely analogue to what we are used to in topology
and if were to be allowed to use the geometric realization functor from graphs to topological 
spaces (of course taking the Whitney complex on the graph as usual so that discrete d-manifolds
go over to smooth compact d-manifolds), the suspension commutes with it. 

\begin{coro}
Minimally chromatic graphs are invariant under suspension.
\label{Theorem1}
\end{coro}
\begin{proof}
All $0$-dimensional graphs are minimally chromatic and especially the $0$-sphere $S_0=P_2$,
the two point graph. The suspension operation $A \to A+S_0$ adds $1$ to both the dimension 
${\rm dim}$ as well as to the chromatic number $X$. 
\end{proof} 

\paragraph{}
The basic compatibility of chromatic number with addition also would help to produce
more examples for which we know the chromatic number. For example: 

\begin{coro}
If there would be a $2$-manifold with $X(G)=6$, then we also had
a $3$-manifolds with $X(G)=7$, a $4$-manifolds with $X(G)=9$ and 
a $5$-manifolds with $X(G)=12$ etc. 
\end{coro}

\paragraph{}
Minimally coloring $d$-manifolds is a generalization of the $4$-color problem because 
coloring $2$-spheres settles the $4$-color problem. In the case $d=2$,
a {\bf 2-manifold} is a finite simple graph for which every unit sphere is a $1$-sphere.
A $1$-sphere $C_n$ is a cyclic graph of length $n=4$ or more.
The task of coloring a manifold is ``simple" after a refinement or more generally after taking
a product: the chromatic number collapses to the minimal number $d+1$. 

\paragraph{}
The Zykov join can be augmented with a multiplication $*$ (called {\bf large multiplication} or 
{\bf Sabidussi multiplication}) to build an associative ring. But
the product of two spheres is never a sphere any more. Still, let us just mention here for
completeness that we still can find the chromatic number $X(A *B)$ of a large product of two 
general graphs $A,B$. 

\begin{lemma}
The chromatic number is multiplicative with respect to the large Sabidussi 
multiplication$X(AB) = X(A) X(B)$.  The chromatic number is therefore a ring homomorphism
from the Sabidussi ring to the integers. 
\end{lemma} 
\begin{proof}
The vertex set of $AB$ is the Cartesian product. Color a point $(a,b)$ with 
the product of the colors of a and b. This shows that we can color $AB$ with 
$X(A) X(B)$ colors. But because every pair $(a,b)$ is connected to $(c,d)$ if 
one of the two $(a,c)$ or $(b,d)$ are connected in $A$ or $B$, we can not 
use less colors. 
\end{proof}

\paragraph{}
One can see this also by {\bf duality} to which we come momentarily.
The graph complement operation is an {\bf symmetry} 
of the category of graphs. It switches the numerical quantities {\bf chromatic number} with
{\bf clique number} as well as switches {\bf independence number}
to {\bf clique covering number} and switches the {\bf disjoint union}
with the {\bf Zykov join} and switches the 
{\bf strong Shannon multiplication} with the {\bf large Sabidussi multiplication}.

\paragraph{}
There are two isomorphic rings, the {\bf Shannon ring} (with disjoint union as addition and 
strong multiplication) and the {\bf Sabidussi ring} (with join as addition and large multiplication). 
The {\bf independence number} $\alpha(G)$ is dual to the 
{\bf clique number} $c(G) = \overline{\alpha}(G)$. The 
{\bf chromatic number} $X$ is dual to the {\bf clique covering number} $\overline{X}$. 
For a multiplicative number in one ring, one can look at the growth rate in the dual picture. The 
clique number $c(G)$ and clique covering number $d(G)$ is multiplicative in the Shannon ring and the 
chromatic number $X(G)$ and independence number $\alpha(G)$ are multiplicative in the Sabidussi ring. 
The exponential growth rate of the independence number therefore is interesting in the Shannon
ring which was the original motivation of Shannon. We can also look at the exponential growth rate
of the chromatic number $\limsup_{n \to infty} X(G^n)^{1/n}$ in the Shannon ring. This is motivated
just by analogy because the independence number is dual to the clique number (which is multiplicative
in the Shannon ring) and the chromatic number which is dual to the clique covering number which 
is multiplicative in the Shannon ring). 

\begin{coro}
$X(A*B) \geq X(A) X(B)$ for the Shannon product. 
\end{coro}
\begin{proof}
The Shannon product graph $A*B$ is a sub-graph of the Sabidussi product graph $AB$ 
so that the chromatic number can only become larger or stay the same. 
\end{proof}

\paragraph{}
One can now ask about the {\bf exponential growth rate} of the chromatic number in the 
Shannon ring
$$  \Theta(G) = \limsup_{n \to \infty} X(G^n)^{1/n} \; .  $$
This is already tough for products of odd cycles. 
For the product of $p$ odd cycles of lengths at least $2^p+1$
one has $\chi(G) = 2^p+1$ \cite{Zerovnik}. 

\paragraph{}
For $G=K_m$ for example, $\Theta(G)$ is $(m+1)$. For $G=C_4$ we still have
$X(C_4*C_4)=4$ but for $G=C_5$ we expect $X(C_5*C_5)$ to be larger than $9$. 
We were unable to have our computer algebra system to evaluate the chromatic 
number of the Shannon product $C_5*C_5$ yet which  is a graph with $f$-vector
$(25, 100, 100, 25)$ and Euler characteristic $0$. It is a graph homotopic to 
a $2$-torus and has Betti vector $(1,2,1)$.
% ChromaticNumber[StrongProduct[CycleGraph[5],CycleGraph[5]]]

\section{Minimal chromatic number}

\paragraph{}
The problem to minimally color a $d$-manifold with $(d+1)$ colors is a problem originally studied
by Percy Heawood. For $2$-manifolds, the Fisk set $O(G)$ of vertices with odd cardinality produces 
{\bf local obstacles} for minimal $3$-coloring. In the simply connected case, there are also global 
constraints which matter. There are also interesting connections between {\bf local and global situations}
as the Fisk set $O(G)$ of odd degree vertices on a discrete $2$-torus can never be a $0$-sphere if the rest is flat:
by Gauss-Bonnet, it would have to consist of a pair of vertices with vertex degree $5$ or $7$ \cite{IKRSS}
which is not possible. The proof is that the Burger's vector is not zero. In the monograph \cite{Fisk1977a}, 
Steve Fisk studied in particular the set of minimal colorings of a simplicial complex theoretically. 

\paragraph{}
Given two finite simple graphs $A,B$, the {\bf Cartesian product} (Stanley-Reisner) 
$(A \times B)_1$ is the graph in which the 
vertices are the pairs $(x,y)$, where $x,y$ are complete sub-graphs of $A$ or $B$ and where 
$(x,y)$ and $(u,v)$ are connected if they are different and either $x \subset u, y \subset v$
or $u \subset x, v \subset y$. If $1$ is the one-point graph, then the product $A \times 1$ 
is the Barycentric refinement of $A$. 

\paragraph{}
This Cartesian product is the Stanley-Reisner ring construction if one looks at the product
algebraically. We were actually searching with the help of a computer for such a product in 2015 
\cite{KnillKuenneth}, then realizing the 
Stanley-Reisner ring construction and then later also saw that it is just the Barycentric refinement
graph of the Cartesian product of simplicial complexes (which is not a simplicial complex but which 
has a Barycentric refinement that is). The actual Stanley-Reisner product is defined
within the polynomial rings. 

\paragraph{}
Like on the level of polynomials, the Stanley-Reisner product given b $A_1,B_1 \to (A \times B)_1$ is
associative. But if we start with graphs $A,B$ and get to $(A \times B)_1$ then it is of course not. 
We have  $A \times (1 \times 1) = A \times 1= A_1$ and $(A \times 1) \times 1 = A_1 \times 1 = A_2$
are first and second Barycentric refinements $A_1,A_2$ are not the same for graphs of positive
dimension. On the level of connection graphs, the Cartesian product becomes
an associative product and leads to the Shannon ring \cite{Shannon1956} (where the disjoint union
is the addition) which is isomorphic to the Sabidussi ring \cite{Sabidussi} (where the join operation
is the addition). As we are interested here primarily in products which preserve manifolds, the
product $(A \times B)_1$ is adequate. The chromatology of this product is not interesting however
as in general a product always produces manifolds with minimal chromatic numbers: 

\begin{lemma}
The Barycentric refinement of any $d$-manifold always has minimal chromatic number $d+1$.
More generally, the product $(A \times B)_1$ of a $p$-manifold and a $q$-manifold is 
a $(p+q)$-manifold of minimal chromatic number $X(A \times B) = p+q+1$. 
\end{lemma}
\begin{proof}
The Barycentric refinement of a graph $A \times 1$ has as vertices the complete subgraphs and connects two
if one is contained in the other. The dimension is the coloring function. The dimension ranges 
from $0$ to the maximal dimension $d$. This means that there are $d+1$ different values. 
\end{proof} 

\paragraph{}
The general problem of coloring $d$-manifolds is hard for the simple reason that already 
the coloring of 2-spheres is equivalent to the $4$-color theorem. 
The fact that $2$-manifolds can have chromatic number $5$ was first demonstrated by Fisk.
It has been conjectured in a similar setting by Albertson-Stromquist that $5$ 
is an upper bound for $2$-manifolds \cite{AlbertsonStromquist}.
Albertson and Stromquist establish this in the case of tori as along as homotopically 
nontrivial cycles have length $8$ or more.

\paragraph{}
An upper bound for the chromatic number in the case of $(d=2)$-manifolds also could be deduced
from the 4-color theorem: the chromatic number of a $2$-manifold with or without boundary 
is $6$ or less because we can make cuts of even length along finitely many curves to 
have connected components which are all planar. 
By the 4-color theorem one can color each component with $4$ colors different from
the $2$ colors which are needed for the cuts. 
For $d=3$, one could cut the 3-manifold into
connected components and get $X \leq 9$ for 3-manifolds. Theorem~(\ref{Theorem1}) 
is already better however there. 

\section{The Fisk variety $O(M)$}

\paragraph{}
A {\bf $(d-2)$-simplex} $x=(x_0,\dots,x_{d-2})$ in a $d$-manifold $G$ defines the $1$-sphere 
$S(x_0) \cap \dots \cap S(x_{d-2})$. It is called the {\bf dual sphere} of $x$. 
It is labeled {\bf odd} if it has odd length. Since Fisk first considered this set, 
we call $O(M)$ the {\bf Fisk variety}. It is not a $(d-2)$-manifold in general. We actually wondered
under which conditions it is a manifold and have an answer below. 

\paragraph{}
The following Lemma essentially goes back to Heawood and is discussed also in \cite{Fisk1977b}.
It only applies for $d$-spheres, where we do not have to worry about monodromy 
issues when coloring.

\begin{lemma}[Heawood] 
For a $d$-sphere, $X(G)=d-1$ if and only if $O(G)=\emptyset$. 
The same holds for $d$-manifolds which are simply connected. 
\end{lemma} 
\begin{proof}
$X(G)=d-1$ means that coloring one simplex determines the
coloring of the entire graph. In order to have compatibility, we need that going along
a closed loop in the dual graph works and that needs each small closed loop to be even. 
\end{proof} 

\paragraph{}
An interesting example of a positive curvature manifold $G$ different from a 
$4$-sphere is the {\bf complex projective plane} $\mathbb{PC}^2$ which is a real 
$4$-manifold and admits a positive curvature metric. (Of course we still mean here a discrete
graph implementing that 4-manifold and not the actual manifold in differential geometry). 
It would be interesting to know
what the maximal chromatic number can be. We expect it to be the same
than the maximal chromatic number of a $4$-sphere: as the Betti vector is $(1,0,1,0,1)$
and $G$ is simply connected, the usual monodromy constraint does not apply. There could
be a surprise however. 

\paragraph{}
For any $2$-manifold $M$, the Fisk set $O(M)$ is a finite set of vertices.
There are an even number of elements in $O(m)$ because 
the {\bf Euler handshake formula} tells $\sum_{x \in V} |S(x)| = 2 |E|$ for any 
finite simple graph $G=(V,E)$. If there would be an odd number of odd degree vertices,
the sum of the vertex degrees would be odd, which is not possible. 
We can rephrase this by saying that $O(M)$ is a finite union of disjoint $0$-spheres. 

\paragraph{}
Let us now look at the case $d=3$: 

\begin{lemma}
For a $3$-manifold $M$, the set $O(M)$ is a finite union of closed curves. They can 
intersect only in $0$-dimensional parts. In particular, there can not be any lose ends. 
If $O(S(x))$ is a $0$-sphere or empty for every unit sphere $S(x)$, 
then $O(M)$ is a simple closed curve.
\end{lemma}

\begin{proof}
The critical edge set intersects at every vertex in an even number of edges. This follows
from the fact that an interior edges in a unit sphere is critical if and only if the
vertex on $S(x)$ is critical. 
\end{proof}

\paragraph{}
The example of the $600$-cell $M$ (which is a $3$-sphere) shows that all edges can 
be part of the Fisk set $O(M)$. The reason is that every vertex in every unit sphere, an
icosahedron has an odd number of $3$ dimensional tetrahedra hinging on. 

\paragraph{}
The following result appears as Proposition~56 in \cite{Fisk1977b} (Fisk did not exactly 
look at the discrete manifold notion considered here but that is not so relevant.)

\begin{lemma}
If $G$ is a $p$-manifold and $H$ is a $q$-manifold with $p,q\geq 2$, then 
$O(G+H) = G+O(H) \cup O(G) + H$. 
\end{lemma}
\begin{proof}
Every $(d-2)$-simplex $x$ in $G+H$ is either 
the join of a $(p-2)$-simplex in $G$ with a $q$ simples $H$ or the 
join $x+y$ of a $q-2$ simplex $x$ in $H$ with a $p$ simplex $y$ in $G$ 
In both cases, $x+y$ is in $O(G+H)$ if the $p-2$ simplex is in $O(H)$ or $O(G)$. 
A third possibility is that $x+y$ is the join of $(p-1)$-simplex 
$x$ in $G$ with a $q-1$ simplex $y$ in $H$ but then the dual $x+y$ 
is the join of two $0$-spheres which makes it even. 
\end{proof}

\paragraph{}
Fisk shows that if a $3$-sphere the two circles can be linked. 
An example is the 3-sphere $G=C_5 + C_5$. An association to the 
{\bf Hopf fibration} comes up. 
Fisk also shows that it is possible to realize any possible knot 
in $S_3$ as $O(M)$.  This is an interesting combinatorial result
as we take a knot, enclose it with $3$-dimensional tetrahedra in such a way
that every edge has an odd number of such tetrahedra hinging on. Now we
argue that we can continue this to a consistent triangulation of the entire
surrounding $3$-sphere. In Fisks proof Seifert surfaces play a crucial role.
One of the open questions (Problem 9) of Fisk asks whether there is 
an example, where $O(G)$ consists of two linked circles of even length.

\paragraph{}
For us, it had been an interesting question is whether it is possible to {\bf edge refine} a
$3$-sphere $M$ with an inscribed {\bf Fisk knot} $O(M)$ to make the knot $O(M)$ disappear after
the modification. We have noticed no constraints so far. 
In the case $C_5 + C_5$ for example, where we have a link of two circles, 
a first refinement merges the two linked circles of length $5$ to a single circle
of length $6$, a second refinement, then removes that circle. As explained several times already, the
ability to edge refine the interior of a 3-ball implies the 4-color theorem. 
An easier still open question is: is it true that every $3$-sphere can be edge refined to become Eulerian?
In the case $d=2$, the answer was yes and we were able to use a billiard cutting procedure to make 
$M$ Eulerian.

\section{The Fisk manifold}

\paragraph{}
A co-dimension $2$-sphere in a sphere is called a {\bf knot}.
More generally, one can also look at co-dimension $2$-manifolds in a manifold and still call them 
{\bf knots}. They play a role in many different parts of mathematics. Just a year ago, we looked
at the structure of positive curvature manifolds \cite{GroveSearle2020} in the {\bf Grove-Searle case}
\cite{GroveSearle2020}, where the theory of Kobayashi \cite{Kobayashi1972} and 
Conner \cite{Conner1957} assures that any positive curvature manifold $M$ with a $S^1$-
symmetry has a {\bf fixed point set} $F(M)$ of even co-dimension. A theorem of Grove-Searle 
\cite{GroveSearle}  then severely restricts the structure of $M$ if $F(M)$ has co-dimension $2$.
It quite directly implies (without further hard analysis) 
that a positive curvature $2,4,6$ or $8$ manifold of that type
has positive Euler characteristic. 

\paragraph{}
Now, in a completely different set-up of discrete manifolds $M$, one has a co-dimension-2 set
$O(M)$ which can be seen as a union of discrete $(d-2)$-manifolds. 
This set has first been considered by Fisk, who also used the notation $O(M)$. We call it the 
{\bf Fisk set} for graphs and {\bf Fisk variety} if $M$ is a discrete manifold.  It sometimes is a 
{\bf manifold} and would call it then the {\bf Fisk manifold}. Fisk had been looked at 
$3$-manifolds, where $O(M)$ was a {\bf link} or a {\bf  knot}. 
An example is the $3$-sphere $G=C_5 + C_5$ where $O(G)$ 
is a union of two circles. If we take $G=C_5 + C_4$, then $O(G)$ is a single circle. 
For the 4-manifold $G=S_0+C_5+C_4$ which is the join of the previous example, then $O(G)$ is 
the suspension of $C_5$, a 2-manifold with $S_0=( \{a,b\},\emptyset)$. In this case,
we have $O(S(a))=O(S_B)=C_5$. For the other unit spheres, we either have $O(S(x))=C_4$ or 
$O(S(x))=0$. 

\paragraph{}
Here is a result which characterizes that the Fisk set $O(M)$ is a manifold:

\begin{thm}
If $O(S(x))$ is a $(d-3)$-sphere in the $(d-1)$-sphere $S(x)$ for all $x$ of $O(S(x))$ is empty,
then $O(M)$ is a $(d-2)$-manifold in $M$. The statement can be reversed.
\end{thm}

\begin{proof}
If $v_0$ is a vertex in $O(M)$ and $x=(v_0,\dots,v_{d-2})$ is
a maximal $(d-2)$-simplex attached to $v$, then
by definition, the intersection $S(v_1) \cap \cdots \cap S(v_{d-2})$ is
an odd circle in $S(v_0)$ because $S(v_0) \cap S(v_1) \cap \cdots \cap S(v_{d-2})$
is an odd circle in $M$. The unit sphere of $x$ in $O(M)$ is now the
set $O(S(x))$ which is a $d-3$ sphere proving that $O(M)$ is a manifold. 
\end{proof}

\paragraph{}
The interesting thing is that $O(G)$ can be a $2$-manifold in a $4$-manifold $G$
such that $O(S(x))=O(M) \cap S(x)$ is a knot in $S(x)$ for some $x$. 

\section{Coloring 2-spheres}

\paragraph{}
In this section we review the restatement of the $4$-color theorem as the 
statement $\chi(G) \leq 4$ for 2-spheres. This has been done before, 
but now we try to avoid the classical notion of {\bf planar} and 
use the {\bf Kuratowski condition} for planarity as a definition. 
The restatement of the 4-color theorem in terms of 2-spheres is 
more elegant because we do not have to bother
with the notion of planarity which involves conditions about embedded graphs. 

\paragraph{}
The {\bf $4$-color theorem} tells that the chromatic number of a 
planar graph is $4$ or less. We can by Kuratowski ditch the Euclidean reference 
and stay within combinatorics by defining
planar as a graph which does not contain a 1D-refined version of 
the hyper-tetrahedron $K_5$ or the utility graph $P_3 + P_3$. 

\paragraph{}
A $2$-sphere is a finite simple graph such that every
unit sphere $S(x)$ is a cyclic graph of length $4$ or more and such that removing a vertex
produces a contractible graph. Recall that a graph is {\bf contractible} if there exists a vertex 
$v \in V$  such that $G$ and $G-v$ are both contractible. 
The following result readily follows from the 4-color theorem because $2$-spheres are planar: 

\begin{thm}[Sphere coloring theorem] 
A $2$-sphere has chromatic number $3$ or $4$.
\end{thm}

\paragraph{}
This sphere version is actually {\bf equivalent} to the standard $4$ color theorem because
of the following two lemmas:

\begin{lemma}[Whitney] 
$G$ is a $2$-sphere if and only if $G$ is maximally planar, is 
$4$-connected and has more than 5 vertices.
\end{lemma}

Proof: (i) Assume first that $G$ is a 2-sphere. Then it must contain a vertex $x$ and a unit sphere $S(x)$
leading to a wheel graph $W$ with at least 5 vertices and then an additional vertex $y$ so that $G-y$ is $W$.
This means that $G$ has at least $6$ vertices. It is planar because $G-v$ is a $2$-ball, a finite simple graph
in which every unit sphere is either a cyclic graph or then a linear graph of length $2$ or more. 
By Kuratowski, we must show that $G$ can not contain 1D-refinement of $K_5$ or $K_{3,3}$. We prove that by 
induction. Start with the smallest possible 2-ball, the wheel graph. This clearly does not contain neither $K_5$
nor $K_{3,3}$ as we can list all sub-graphs. Now verify the Kuratowski conditions by induction. Assume it works for
$n$ and take a $2$-ball $G'$ of order $n+1$. It is of the form $G' = G +_A x$ where $x$ is attached to a linear 
subgraph $A$ the boundary circle of $G$. 
If $G'$ contained an embedded utility graph $H+x$, then also $H$ had an embedded utility graph.
If $G'$ contained an embedded complete graph $H+x$, then also $H$ had an embedded complete graph $H$. 
Now show that a 2-sphere is maximally planar: adding an edge $e=(a,b)$ adds a new point $b$ to a unit sphere 
$S(a)$ which is then no more a circle: either $b$ is part of $S(a)$ or not and in both cases the sphere $S_{S(a)}(a)$
has more than $2$ points. 
Finally, we verify that a 2-sphere $G$ is 4-connected. Assume $G$ could be separable by three points $(a,b,c)$. We can 
assume that two are connected as removing an isolated point does not change connectivity. Either $(a,b,c)$ is a linear
graph or then a triangle. In both cases, the complement is connected as can be shown by verifying with respect to 
induction on the size of $G-x$.  \\

      (ii) now assume G is maximally planar and 4 connected and has at least 6 vertices.
Each unit sphere must have 4 or more elements as otherwise it would not be 4-connected. 
The unit sphere can not contain $K_3$ as we otherwise have $G=K_4$ or that $G$ contains
a $K_4$ as a strict subgraph, violating 4-connectivity.
As $S(x)$ has no triangles, it is one dimensional.
The vertex degree is 2 for every point in S(x). If it is 1 for some point $a$ in $S(x)$, 
then there must be an other end point $b$ in $S(x)$ and we can connect $(a,b)$
without violating 4 connectivity nor planarity and so violate maximal planarity. 
If the vertex degree were larger than 2, and the neighbors are a,b,c, this means that S(x) contains a star 
graph y-a,y-b,y-z. Removing x,y and one of the vertices a,b,c must keep the graph connected so that all 
a,b,c are connected. But that means that x,y,a,b,c is homeomorphic to a complete graph $K_5$ which violates
the Kuratowski definition of planarity. We know now $S(x)$ is a circular graph and so that 
G is a discrete manifold. The analogue classification of $2$-manifolds holds in the discrete. The orientability
and the genus determines the type. 
In the positive genus case, G is not simply connected, we can embed the utility graph and G is not planar.
In any 2-manifold we can embed $K_{2,3}$ which has $5$ points $P_3 + P_2$. 
Now, if there is an additional homotopically non-trivial closed loop, we can embed $K_{3,3}$ by 
connecting a 6'th point to $P_3$ along the loop.

\paragraph{}
Whitney writes in \cite{WhitneyCollected} about a verson of this 
result and expressed satisfaction that it can be used to prove the Kuratowski theorem. 
      
\paragraph{} 
The second lemma is hard to attribute. According to \cite{Stromquist}, it is a reduction 
step has been rediscovered by virtually anybody working on graph coloring: 

\begin{lemma}[Folklore]
If we can color maximally planar $4$-connected graphs, we can color all planar graphs. 
\end{lemma}

\begin{proof}
Given a planar graph $H$. Make it maximal. If we can color the maximal one, we can 
color $H$. So, we can assume $G$ is maximally planar. Every graph can be decomposed
into $4$-connected pieces. If there is one $4$-connected component, we have a 2-sphere.
If there are n connected $4$-components, make a cut for one.
\end{proof}

\paragraph{}
We originally thought that coloring $2$-spheres is easy
by using Kempe chains along a Reeb foliation \cite{knillreeb}. 
{\bf Kempe chains} famously work well for 
the $5$-color theorem but it is not constructive there as it is a reduction argument
establishing that a degree $5$ vertex in a minimal counter example needing $6$ colors
can not exist when coloring a $2$-sphere. The explicit coloring for the 
upper bound $X(G) \leq 2d+2$ for all $d$-manifolds is constructive in all dimensions and
in particular for $d=2$, where just cut the dual graph of the $2$-sphere. This is a $3$-regular
triangle-free graph. Let us just state this for the record:

\begin{coro}
The process of 6-coloring a $2$-manifold is constructive and can be done fast. 
The same holds for $(2d+2)$-coloring a $d$-manifold. 
\end{coro}

\paragraph{}
The Fisk story indicates that it should be possible to $4$-color a $2$-sphere $S$
constructively in polynomial time by implementing the algorithm to minimally 
cover a $3$-ball having $S$ as the boundary. This would especially allow to 4-color
in polynomial time any planar graph with $n$ vertices constructively.

\bibliographystyle{plain}

\end{document}